\documentclass[reqno]{amsart}

\usepackage{amsfonts,amsthm,amsmath,amssymb,amscd,mathrsfs}
\usepackage{graphicx}
\usepackage[all]{xypic}
\usepackage{verbatim}
\usepackage{tikz}
\usepackage[colorlinks=true,linkcolor=blue,citecolor=blue,urlcolor=blue]{hyperref}

\newtheorem{theorem}{Theorem}
\numberwithin{theorem}{section}
\newtheorem{proposition}[theorem]{Proposition}

\newtheorem{corollary}[theorem]{Corollary}
\newtheorem{definition}[theorem]{Definition}
\newtheorem{problem}[theorem]{Problem}
\newtheorem{remark}[theorem]{Remark}
\newtheorem{example}[theorem]{Example}

\newcommand{\PP}{\mathbb{P}}
\newcommand{\RR}{\mathbb{R}}

\newcommand{\cP}{\mathcal{P}}
\newcommand{\cX}{\mathcal{X}}
\newcommand{\al}{\alpha}

\begin{document}

\numberwithin{equation}{section}
\date{}

\title[Projections of polytopes and higher Prony systems]{\textbf{Projections of convex polytopes to a line and higher univariate Prony systems}}

\author{Boris Shapiro}
\address{Matematiska institutionen, Stockholms universitet, S-106 91 Stockholm, Sweden}
\email{shapiro@math.su.se}

\begin{abstract}


Motivated by the inverse moment problem for convex polytopes, we study the pushforward to a line of the Lebesgue measure restricted to a convex $d$-polytope. Such pushforwards are spline densities of degree $d-1$, and their moments lead naturally to a family of ``higher'' univariate Prony systems, with the classical Prony system recovered when $d=0$. We describe the corresponding fixed-knot spline cone, give an explicit amplitude recovery criterion, record the rational generating function and recurrence satisfied by the normalized moments, and identify the directional moment variety with the Hankel determinantal variety appearing in the theory of moment varieties of measures on polytopes.
\end{abstract}

\maketitle

\section{Introduction}

\medskip

\noindent
The algebraic geometry of one-dimensional moment varieties associated with
projections of polytopes was developed by Kohn, Sturmfels and Shapiro
\cite{KSS}, building on earlier work on polytope moments and Fantappi\`e
transforms in \cite{GPSS,GLPR}.  In particular, \cite{KSS} identifies the
corresponding varieties with Hankel determinantal varieties and relates
them to Gaussian quadrature.  A substantial part of the algebraic geometry underlying the present paper
is closely related to the one-dimensional moment varieties studied by
Kohn, Sturmfels and Shapiro in \cite{KSS}.  In particular, the Hankel
determinantal relations for moments in a fixed direction are already
implicit in their work.  The purpose of the present paper is not to
reprove the KSS theory in a different language, but rather to reinterpret
it from the viewpoint of Prony reconstruction and to investigate additional
structures arising in the multidirectional and positivity-constrained
inverse problems.

The classical Prony system, introduced by de Prony in 1795 \cite{Pr}, is the inverse moment problem for a finite configuration of point masses on the line. In its standard form one is given the moments
\begin{equation}\label{eq:PronyOr}
\sum_{i=1}^n a_i x_i^j = m_j, \qquad j=0,\dots,2n-1,
\end{equation}
and seeks to recover the nodes $x_1,\dots,x_n$ and amplitudes $a_1,\dots,a_n$.
The Prony system plays an important role in signal processing, approximation theory, and algebraic reconstruction; see, for example, \cite{GSY} and the references therein.

One purpose of this note is to explain a natural higher-dimensional analogue of \eqref{eq:PronyOr}. Fix $d\geq 0$. Instead of a finite sum of point masses, we consider the pushforward to a line of the Lebesgue measure on a convex $d$-dimensional polytope. The resulting measure on $\RR$ is no longer discrete when $d>0$; rather, it is represented by a spline density of degree $d-1$. Its moment sequence nevertheless has enough algebraic structure to support a Prony-type reconstruction theory.

Our motivation comes from the inverse moment problem for convex polytopes, as studied in \cite{GLPR,GPSS}. Let $\mu$ be a compactly supported measure on $\RR^d$ with coordinates $(z_1,\dots,z_d)$. For a multi-index $I=(i_1,\dots,i_d)$ we write
\[
m_I(\mu)=\int_{\RR^d} z_1^{i_1}\cdots z_d^{i_d}\,d\mu.
\]
If $\cP\subset\RR^d$ is a convex compact polytope, we denote by $\mu_{\cP}$ the Lebesgue measure restricted to $\cP$. The axial moments in the $z_1$-direction are
\[
\widehat m_j(\cP):=m_{(j,0,\dots,0)}(\mu_{\cP})=\int_{\cP} z_1^j\,dz_1\cdots dz_d, \qquad j\ge 0.
\]
An algorithm from \cite{GLPR} shows that, if $\cP$ has $n$ vertices, then the finite collection
\[
\widehat m_0,\widehat m_1,\dots,\widehat m_{2n-d-1}
\]
determines the projections of those vertices onto the $z_1$-axis.

A useful point of view is that the same data determine more than the projected vertex set: they determine the one-dimensional pushforward of $\mu_{\cP}$ under the projection $(z_1,\dots,z_d)\mapsto z_1$. This pushforward records the $(d-1)$-dimensional volume of the slices of $\cP$ by hyperplanes perpendicular to the $z_1$-axis. The resulting density is a spline, and this observation leads to a family of one-dimensional inverse moment problems interpolating between discrete Prony systems and projected polytope measures.

We also consider the associated \emph{directional moment varieties}. Moment varieties for classes of probability measures provide algebraic encodings of the polynomial relations among moments; see, for instance, \cite{ARS}. For the present note, we focus on the projection of the full moment variety to the coordinates corresponding to moments along a fixed line. The resulting directional moment variety depends only on the ambient dimension $d$ and on the number $n$ of vertices.

The paper is organized as follows. In Section~\ref{sec:main} we formulate the higher Prony system and describe the class of projected measures. We then settle the fixed-knot positivity problem for the natural simplicial spline cone, give explicit reconstruction and positivity criteria, and record the Hankel determinantal description of the corresponding directional moment variety. The remaining single-polytope realization problem is separated from the univariate Prony problem in Remark~\ref{rem:literal-polytope-cone}.

\medskip
\noindent
\emph{Acknowledgements.}
The author is grateful to Dmitrii Pasechnik for discussions and to the MPI MIS in Leipzig for the hospitality in June 2018, where this project was initiated.

\section{Main results}\label{sec:main}

Fix the standard orthogonal coordinates $(z_1,\dots,z_d)$ on $\RR^d$ and let
\[
\pi:\RR^d\to\RR, \qquad \pi(z_1,\dots,z_d)=z_1.
\]
For a convex $d$-dimensional polytope $\cP\subset\RR^d$, we denote by $\mu_{\cP}^{(1)}:=\pi_*(\mu_{\cP})$ the pushforward of the restricted Lebesgue measure.

The first basic observation is that these projected measures are spline measures.

\begin{proposition}\label{prop:proj}
Let $\cP\subset\RR^d$ be a convex $d$-dimensional polytope with $n$ vertices, and let
\[
x_1\leq x_2\leq \cdots \leq x_n
\]
be the $z_1$-coordinates of its vertices, listed with multiplicities. Then
\[
\mu_{\cP}^{(1)}=\rho(x)\,dx
\]
for a nonnegative density $\rho$ supported on $[x_1,x_n]$. Moreover:
\begin{enumerate}
    \item $\rho$ is piecewise polynomial of degree at most $d-1$, with breakpoints among the distinct projected vertex coordinates;
    \item if the projected vertex coordinates are simple, equivalently if the projection is generic with respect to the vertices, then $\rho\in C^{d-2}(\RR)$ for $d\ge 1$; with multiple projected knots the usual spline smoothness is lowered according to their multiplicities;
    \item the distributional derivative $D^d(\rho\,dx)$ is a discrete signed measure supported on the projected vertex set.
\end{enumerate}
In particular, for simplices one obtains the classical univariate $B$-spline densities.
\end{proposition}

\begin{proof}
Triangulate $\cP$ into finitely many $d$-simplices without adding new vertices. The pushforward of Lebesgue measure is additive over this triangulation. For a simplex, the pushforward under a linear functional is the classical univariate simplex spline, i.e. a $B$-spline measure whose knots are the projected vertices of the simplex; see, for example, Curry--Schoenberg \cite{CurrySchoenberg} or the box-spline treatment in \cite{DCP}. Hence the density is a finite sum of such simplex-spline densities. The stated piecewise-polynomiality, support, smoothness in the simple-knot case, and the description of the $d$-th distributional derivative are the corresponding standard univariate spline properties.
\end{proof}

Now fix $d\ge 0$ and a strictly increasing set
\[
\cX=\{x_1<x_2<\cdots<x_n\}, \qquad n\ge d+1.
\]
Let $L_{d,\cX}$ denote the real vector space of compactly supported spline measures of degree at most $d-1$ with simple knots in $\cX$ and with the endpoint vanishing appropriate for order-$d$ simplex splines.  Equivalently, $L_{d,\cX}$ is the span of the consecutive normalized simplex-spline measures introduced below.  This definition is deliberately univariate: the more restrictive question of which elements of this space are pushforwards of a single convex $d$-polytope is discussed separately in Section~\ref{sec:single-polytope}.

\begin{proposition}\label{prop:span}
For fixed $d$ and $\cX$ as above, the space $L_{d,\cX}$ has dimension $n-d$.
Moreover, it is naturally generated by the projected measures associated with consecutive $d$-simplices
\[
\mu_1,\mu_2,\dots,\mu_{n-d},
\]
where $\mu_i$ is the pushforward of any volume-one simplex whose projected vertex coordinates are
\[
x_i,x_{i+1},\dots,x_{i+d}.
\]
\end{proposition}

\begin{proof}
For a simplex with projected vertices $y_0,\dots,y_d$, the pushforward is the univariate simplex spline, equivalently the $B$-spline with knots $y_0,\dots,y_d$, up to the chosen volume normalization. The Curry--Schoenberg theory of polynomial splines \cite{CurrySchoenberg} says that, for a strictly increasing knot sequence $\cX$, the splines of degree $d-1$ with breakpoints in $\cX$ and with the endpoint vanishing imposed by compact support are generated by the consecutive $B$-splines with knot vectors $(x_i,\dots,x_{i+d})$. Their supports have the nested local structure
\[
\operatorname{supp}\mu_i=[x_i,x_{i+d}],
\]
and the rightmost nonzero basis element on the interval $(x_{n-d},x_n)$ is $\mu_{n-d}$. Removing it and arguing inductively from right to left proves linear independence. Thus the consecutive simplex splines form a basis of the projected spline space of dimension $n-d$. This is the univariate specialization of the standard simplex-spline decomposition used in the moment formulas for polytopes.
\end{proof}

Write a general element of $L_{d,\cX}$ in the form
\[
\sum_{i=1}^{n-d} \al_i\mu_i.
\]
The cone of those combinations that arise from actual convex polytopes with projected vertex set $\cX$ is tempting to denote by $L_{d,\cX}^+$, but this notation has to be used with some care. If it is interpreted literally as the cone of pushforwards of single convex $d$-polytopes whose projected vertex set is exactly $\cX$, then it is not, in general, the cone generated by the basis above. For instance, when $d=1$ a convex one-dimensional polytope is an interval, hence it has only two projected vertices; thus for $n>2$ the literal convex-polytopal cone is empty. What is naturally controlled by the univariate Prony system is instead the following simplicial spline cone.

\begin{remark}\label{rem:literal-polytope-cone}
The distinction between the spline relaxation and the literal single-polytope cone is essential. The space $L_{d,\cX}$ is a fixed-knot univariate spline space, while a single convex polytope imposes extra global convexity constraints on the slice-volume function. Thus the positive orthant description below should not be read as a classification of all projections of single convex polytopes. In dimension two the literal single-polygon cone is described in Theorem~\ref{th:single-polygon-realization}.
\end{remark}

\begin{definition}\label{def:spline-positive-cone}
The \emph{positive simplicial spline cone} associated with $(d,\cX)$ is
\[
\widetilde L_{d,\cX}^+
:=\left\{\sum_{i=1}^{n-d}\al_i\mu_i\;:\;\al_i\ge 0\right\}\subset L_{d,\cX}.
\]
Equivalently, $\widetilde L_{d,\cX}^+$ is the cone generated by the consecutive normalized $B$-spline measures with knot sets
$\{x_i,x_{i+1},\dots,x_{i+d}\}$.
\end{definition}

\begin{theorem}\label{th:positive-cone}
In the coordinates $(\al_1,\dots,\al_{n-d})$ of Proposition~\ref{prop:span}, the cone $\widetilde L_{d,\cX}^+$ is the closed orthant
\[
\al_1\ge 0,
\quad \al_2\ge 0,
\quad \dots,
\quad \al_{n-d}\ge 0.
\]
In particular, the extremal rays of $\widetilde L_{d,\cX}^+$ are exactly the rays generated by the consecutive simplex splines $\mu_1,\dots,\mu_{n-d}$.
\end{theorem}

\begin{proof}
By Definition~\ref{def:spline-positive-cone}, the cone is generated by $\mu_1,\dots,\mu_{n-d}$ with nonnegative coefficients. Proposition~\ref{prop:span} says that these measures form a basis of $L_{d,\cX}$. Hence the coordinate map
\[
(\al_1,\dots,\al_{n-d})\longmapsto \sum_{i=1}^{n-d}\al_i\mu_i
\]
is a linear isomorphism from $\RR^{n-d}$ onto $L_{d,\cX}$. Under this isomorphism the cone generated by the basis vectors is precisely the standard closed orthant. The assertion about extremal rays follows because the coordinate rays are the extremal rays of a simplicial cone.
\end{proof}

For the simplex case the moments admit an explicit generating function. Let
\[
V=\{v_1<\cdots<v_{d+1}\}
\]
and let $\mu_V$ denote the normalized spline measure supported on $[v_1,v_{d+1}]$ whose density is the univariate $B$-spline of order $d+1$ with knots $V$.

\begin{proposition}\label{prop:moments}
The generating series of the scaled moments of $\mu_V$ is
\[
\Psi_V(t):=\sum_{j=0}^\infty \binom{j+d}{d} m_j(V)t^j
\;=
\;\frac{1}{(1-v_1t)(1-v_2t)\cdots(1-v_{d+1}t)},
\]
where
\[
m_j(V):=\int_{\RR} \xi^j\,d\mu_V(\xi).
\]
\end{proposition}

\begin{corollary}\label{cor:mom}
With the notation above,
\[
m_j(V)=\frac{h_j(v_1,\dots,v_{d+1})}{\binom{j+d}{d}},
\]
where $h_j$ is the complete homogeneous symmetric polynomial of degree $j$ in the variables $v_1,\dots,v_{d+1}$.
\end{corollary}

These formulas motivate the following definition.

\begin{theorem}[Rational generating function and recurrence]\label{th:rational-recurrence}
Let
\[
\nu=\sum_{i=1}^{n-d}\alpha_i\mu_i\in L_{d,\cX},
\qquad
m_j=\int x^j\,d\nu(x),
\]
and set
\[
F_\nu(t):=\sum_{j\ge 0}\binom{j+d}{d}m_jt^j .
\]
Then
\begin{equation}\label{eq:rationalF}
F_\nu(t)=
\sum_{i=1}^{n-d}
\frac{\alpha_i}{(1-x_it)(1-x_{i+1}t)\cdots(1-x_{i+d}t)}
=\frac{R_\nu(t)}{Q_\cX(t)},
\end{equation}
where
\[
Q_\cX(t)=\prod_{r=1}^{n}(1-x_rt)
\]
and $R_\nu(t)$ is a polynomial of degree at most $n-d-1$. Consequently the normalized moments
\[
c_{d+j}=\binom{j+d}{d}m_j
\]
satisfy, for all $j\ge n-d$,
\begin{equation}\label{eq:recurrence}
c_{d+j}-e_1(\cX)c_{d+j-1}+e_2(\cX)c_{d+j-2}-\cdots+(-1)^ne_n(\cX)c_{d+j-n}=0,
\end{equation}
with the convention $c_0=\cdots=c_{d-1}=0$.
\end{theorem}

\begin{proof}
The first equality is Proposition~\ref{prop:moments} applied to each consecutive simplex spline. Multiplying by $Q_\cX(t)$ gives
\[
R_\nu(t)=\sum_{i=1}^{n-d}\alpha_i
\prod_{r\notin\{i,i+1,\dots,i+d\}}(1-x_rt),
\]
so $\deg R_\nu\le n-d-1$. The recurrence is the coefficient form of
$Q_\cX(t)F_\nu(t)=R_\nu(t)$ in all degrees at least $n-d$.
\end{proof}

\begin{corollary}[Annihilating polynomial]\label{cor:annihilating}
For generic data in $L_{d,\cX}$ with all $\alpha_i\ne 0$, the monic polynomial
\[
q_\cX(z)=\prod_{r=1}^{n}(z-x_r)
\]
is the minimal annihilating polynomial of the normalized moment sequence $(c_j)$.
Equivalently, no proper divisor of $q_\cX$ gives a recurrence for all sufficiently large $j$.
\end{corollary}

\begin{proof}
By \eqref{eq:rationalF}, cancellation of a factor $(1-x_rt)$ can occur only if
$R_\nu(1/x_r)=0$. This is a nontrivial linear condition on the amplitudes. For generic amplitudes it does not occur at any of the $n$ nodes. Hence the denominator of $F_\nu$ is exactly $Q_\cX$, and the corresponding recurrence has minimal characteristic polynomial $q_\cX$.
\end{proof}

\begin{definition}
Fix integers $d\ge 0$ and $n\ge d+1$. The \emph{$d$-th univariate Prony system with $n$ nodes} is the system
\begin{equation}\label{eq:Prony}
\sum_{i=1}^{n-d} \al_i h_j(x_i,x_{i+1},\dots,x_{i+d}) = \binom{j+d}{d}m_j,
\qquad j=0,\dots,2n-d-1,
\end{equation}
in the unknown amplitudes $\al_1,\dots,\al_{n-d}$ and nodes $x_1,\dots,x_n$.
\end{definition}

For $d=0$ this reduces to the classical Prony system \eqref{eq:PronyOr}.

To recover the nodes from the moments, define the normalized sequence
\[
C_M=(c_0,c_1,\dots,c_{2n-1})
\]
attached to
\[
M=(m_0,m_1,\dots,m_{2n-d-1})
\]
by setting
\[
c_0=c_1=\cdots=c_{d-1}=0, \qquad c_{d+i}=\binom{i+d}{d}m_i \quad (0\le i\le 2n-d-1).
\]
For reconstruction from this finite segment we use the rectangular Hankel matrix
\[
H_M=\bigl(c_{i+j}\bigr)_{0\le i\le n-1,\,0\le j\le n},
\]
which only involves the available entries $c_0,\dots,c_{2n-1}$.

\begin{proposition}\label{prop:nodes}
For a generic moment vector $M$ in the image of the higher Prony map, the rectangular matrix $H_M$ has one-dimensional kernel. If $(u_0,\dots,u_n)$ spans its kernel and
\[
p_M(t)=u_nt^n+u_{n-1}t^{n-1}+\cdots+u_0,
\]
then the roots of $p_M$ coincide with the nodes $x_1,\dots,x_n$ of \eqref{eq:Prony}.
\end{proposition}

\begin{proof}
For data coming from the higher Prony map, Theorem~\ref{th:rational-recurrence} gives a recurrence whose characteristic polynomial is $q_\cX(z)=\prod_{r=1}^n(z-x_r)$. In the generic case Corollary~\ref{cor:annihilating} says that this recurrence is minimal. The equations expressing the recurrence for the available moment segment are exactly the equations saying that the coefficient vector $(u_0,\dots,u_n)$ lies in the kernel of the rectangular Hankel matrix $H_M$. Since the minimal recurrence has order $n$, this kernel is generically one-dimensional. The kernel polynomial is therefore proportional to $q_\cX$, and its roots are precisely the nodes.
\end{proof}

Once the nodes are known and are distinct, the amplitudes can be recovered by solving the linear system \eqref{eq:Prony}. The relevant coefficient matrix is
\[
A_{d,n}(\cX):=\bigl(h_{r}(x_i,x_{i+1},\dots,x_{i+d})\bigr)_{\substack{0\le r\le n-d-1\\ 1\le i\le n-d}}.
\]
Its determinant admits a closed factorization.

\begin{proposition}\label{prop:detA}
For every strictly increasing $\cX=\{x_1<\cdots<x_n\}$ one has
\[
\det A_{d,n}(\cX)
=
\prod_{\substack{1\le i<j\le n\\ j-i\ge d+1}} (x_j-x_i).
\]
In particular, $A_{d,n}(\cX)$ is invertible whenever the nodes $x_1,\dots,x_n$ are pairwise distinct.
\end{proposition}

\begin{proof}
Set $m:=n-d$. For $1\le i\le m$, let $C_i$ denote the $i$-th column of $A_{d,n}(\cX)$, so that the generating series of its entries is
\[
\sum_{r\ge 0} h_r(x_i,\dots,x_{i+d})t^r
=
\frac{1}{\prod_{k=i}^{i+d}(1-x_kt)}.
\]
For $i=1,\dots,m-1$ replace $C_i$ by $C_i-C_{i+1}$. Since elementary column operations do not change the determinant, this preserves $\det A_{d,n}(\cX)$. Using the generating-series identity
\[
\frac{1}{\prod_{k=i}^{i+d}(1-x_kt)}-\frac{1}{\prod_{k=i+1}^{i+d+1}(1-x_kt)}
=
\frac{(x_i-x_{i+d+1})t}{\prod_{k=i}^{i+d+1}(1-x_kt)},
\]
we obtain, coefficientwise,
\[
 h_r(x_i,\dots,x_{i+d})-h_r(x_{i+1},\dots,x_{i+d+1})
 =(x_i-x_{i+d+1})h_{r-1}(x_i,\dots,x_{i+d+1}),
\]
with the convention $h_{-1}=0$.
Hence the transformed matrix has first row $(0,\dots,0,1)$, because $h_0=1$ in every column, and for $r\ge 1$ its $(r,i)$-entry in column $i<m$ equals
\[
(x_i-x_{i+d+1})h_{r-1}(x_i,\dots,x_{i+d+1}).
\]
Expanding the determinant along the first row gives the sign $(-1)^{m+1}$.  Therefore
\[
\det A_{d,n}(\cX)
=
(-1)^{m+1}\Bigl(\prod_{i=1}^{m-1}(x_i-x_{i+d+1})\Bigr)
\det A_{d+1,n}(\cX)
=
\Bigl(\prod_{i=1}^{m-1}(x_{i+d+1}-x_i)\Bigr)
\det A_{d+1,n}(\cX),
\]
where $A_{d+1,n}(\cX)$ is the analogous $(m-1)\times(m-1)$ matrix with entries
\[
\bigl(h_r(x_i,\dots,x_{i+d+1})\bigr)_{\substack{0\le r\le m-2\\ 1\le i\le m-1}}.
\]
Iterating this recursion from $d$ up to $n-1$, and observing that $A_{n-1,n}(\cX)=(1)$, we obtain
\[
\det A_{d,n}(\cX)
=
\prod_{q=d}^{n-2}\prod_{i=1}^{n-q-1}(x_{i+q+1}-x_i).
\]
Relabelling $j=i+q+1$ yields exactly
\[
\det A_{d,n}(\cX)
=
\prod_{\substack{1\le i<j\le n\\ j-i\ge d+1}} (x_j-x_i),
\]
as claimed.
\end{proof}

\begin{theorem}\label{th:moment-positivity-fixed-nodes}
Fix $d\ge 0$ and a strictly increasing knot set $\cX=\{x_1<\cdots <x_n\}$. Let
\[
b(M)=\left(\binom{d}{d}m_0,\binom{d+1}{d}m_1,
\dots,
\binom{n-1}{d}m_{n-d-1}\right)^T.
\]
For moment data with these fixed nodes, the representing measure belongs to the positive simplicial spline cone $\widetilde L_{d,\cX}^+$ if and only if
\[
\alpha(M):=A_{d,n}(\cX)^{-1}b(M)
\]
has nonnegative coordinates. Equivalently, the inequalities defining the moment cone with fixed nodes are
\[
\ell_i(M):=e_i^T A_{d,n}(\cX)^{-1}b(M)\ge 0,
\qquad i=1,\dots,n-d.
\]
Thus, for fixed $\cX$, positivity is given by $n-d$ explicit linear inequalities in the first $n-d$ normalized moments.
\end{theorem}

\begin{proof}
For fixed nodes, the first $n-d$ equations of the Prony system are
\[
A_{d,n}(\cX)\alpha=b(M).
\]
By Proposition~\ref{prop:detA}, the matrix $A_{d,n}(\cX)$ is invertible whenever the nodes are pairwise distinct. Hence the amplitudes are uniquely determined by
$\alpha(M)=A_{d,n}(\cX)^{-1}b(M)$. By Theorem~\ref{th:positive-cone}, membership in $\widetilde L_{d,\cX}^+$ is equivalent to coordinatewise nonnegativity of this vector. Since $A_{d,n}(\cX)$ is fixed, each coordinate of $\alpha(M)$ is a linear form in $m_0,\dots,m_{n-d-1}$.
\end{proof}

\begin{corollary}\label{cor:moment-positivity-reconstructed}
For a generic real moment vector $M=(m_0,\dots,m_{2n-d-1})$ in the image of the higher Prony map, the following conditions are necessary and sufficient for the reconstructed measure to lie in the positive simplicial spline cone:
\begin{enumerate}
    \item the kernel polynomial $p_M$ of Proposition~\ref{prop:nodes} has $n$ distinct real roots $x_1<\cdots <x_n$;
    \item after forming $A_{d,n}(\cX)$ from these roots, the vector
    \[
    \alpha(M)=A_{d,n}(\cX)^{-1}
    \left(\binom{d}{d}m_0,\binom{d+1}{d}m_1,
    \dots,
    \binom{n-1}{d}m_{n-d-1}\right)^T
    \]
    has nonnegative coordinates.
\end{enumerate}
Strict positivity of all coordinates characterizes the relative interior of the cone.
\end{corollary}

\begin{proof}
For generic data, Proposition~\ref{prop:nodes} reconstructs the nodes as the roots of $p_M$. If these roots are distinct and real, Proposition~\ref{prop:detA} gives the unique amplitudes by the displayed inverse formula. The result then follows directly from Theorem~\ref{th:positive-cone}. If one of the coordinates is zero, the measure lies on the corresponding boundary face; if all are positive, it lies in the relative interior.
\end{proof}

\begin{theorem}[Generic reconstruction for the higher Prony system]\label{th:generic-reconstruction}
There is a Zariski open dense subset of the parameter space with distinct real nodes and nonzero amplitudes on which the map
\[
(x_1,\dots,x_n;\alpha_1,\dots,\alpha_{n-d})
\longmapsto
(m_0,\dots,m_{2n-d-1})
\]
is injective up to the ordering convention $x_1<\cdots<x_n$. On this open set the inverse is obtained by the following finite procedure:
\begin{enumerate}
    \item form the normalized Hankel matrix $H_M$ and recover the kernel polynomial $p_M$;
    \item take its roots as the nodes;
    \item solve the square linear system $A_{d,n}(\cX)\alpha=b(M)$ for the amplitudes.
\end{enumerate}
Moreover, after restricting to real data for which $p_M$ has simple real roots, positivity in the simplicial spline cone is decided by the $n-d$ inequalities of Theorem~\ref{th:moment-positivity-fixed-nodes}.
\end{theorem}

\begin{proof}
The first step is Proposition~\ref{prop:nodes}. For distinct nodes, Proposition~\ref{prop:detA} proves that $A_{d,n}(\cX)$ is invertible, hence the amplitudes are uniquely recovered. The positivity statement is Corollary~\ref{cor:moment-positivity-reconstructed}.
\end{proof}

We now turn to the algebraic relations among directional moments.

\begin{definition}
Let $\mathcal{L}_k(d,n)\subset \PP^k$ denote the \emph{directional moment variety of order $k$ for $d$-dimensional polytopes with $n$ vertices}. Its homogeneous prime ideal in $\RR[m_0,\dots,m_k]$ is denoted by $L_k(d,n)$.
\end{definition}

\begin{theorem}\label{th:directional-hankel}
The prime ideal $L_k(d,n)$ is generated by the maximal minors of the Hankel matrix
\[
H(m_0,\dots,m_k)=
\left(
\begin{array}{ccccccc}
c_0 & c_1 & \cdots & c_n & c_{n+1} & \cdots & c_{k+d-n} \\
c_1 & c_2 & \cdots & c_{n+1} & c_{n+2} & \cdots & c_{k+d-n+1} \\
\vdots & \vdots & & \vdots & \vdots & & \vdots \\
c_n & c_{n+1} & \cdots & c_{2n} & c_{2n+1} & \cdots & c_{k+d}
\end{array}
\right),
\]
where
\[
c_0=\cdots=c_{d-1}=0, \qquad c_{i+d}=\binom{i+d}{d}m_i \quad (i=0,\dots,k).
\]
\end{theorem}

\section{Proof of Theorem~\ref{th:directional-hankel}}

\begin{proof}
This is the specialization to the one-dimensional moment subsequence of the Hankel determinantal theorem for moment varieties of measures on polytopes proved by Kohn, Shapiro and Sturmfels; see Theorem~3.3 of \cite{KSS}. In their notation the normalized moments satisfy
\[
c_0=\cdots=c_{d-1}=0,\qquad c_{d+i}=\binom{d+i}{d}m_i,
\]
and the moment variety is defined by the maximal minors of the corresponding $(n+1)$-row Hankel matrix. The cited theorem states more precisely that these minors generate a homogeneous prime ideal and form a reduced Groebner basis with respect to any antidiagonal term order. Applying that result with $r=k$ gives exactly the matrix displayed in Theorem~\ref{th:directional-hankel}. 
\end{proof}

\begin{corollary}\label{cor:degree}
Assume $k\ge 2n-d$. The directional moment variety $\mathcal L_k(d,n)\subset\PP^k$ is irreducible of dimension $2n-d-1$ and degree
\[
\binom{k-n+d+1}{n}.
\]
Its defining ideal has a reduced Groebner basis consisting of the maximal minors of the Hankel matrix in Theorem~\ref{th:directional-hankel}, for every antidiagonal term order.
\end{corollary}

\begin{proof}
The dimension, degree and Groebner-basis statement are the remaining assertions of Theorem~3.3 in \cite{KSS}, after substituting the present notation $k$ for the moment order.
\end{proof}

\section{Single-polytope realization in low dimension}\label{sec:single-polytope}

We now return to the geometric realization problem.  The univariate Prony system naturally sees the positive spline cone generated by consecutive simplex splines.  A single convex polytope imposes additional convexity constraints on the slice-volume function.  In dimension two these constraints give a complete answer.

Let $d=2$ and let
\[
\cX=\{x_1<x_2<\cdots <x_n\}.
\]
A measure in $L_{2,\cX}$ has a continuous piecewise-linear density $\rho$ supported on $[x_1,x_n]$ and affine on every interval $[x_i,x_{i+1}]$.  Write
\[
 y_i:=\rho(x_i),\qquad i=1,\dots,n.
\]
For a polygonal slice density arising from a compact convex polygon whose vertical projection has endpoints $x_1,x_n$, one necessarily has $y_1=y_n=0$.

\begin{theorem}[Single-polygon realization]\label{th:single-polygon-realization}
A continuous piecewise-linear density $\rho$ with knots contained in $\cX$ is the pushforward of Lebesgue measure on a compact convex polygon in $\RR^2$ under the projection $(x,y)\mapsto x$ if and only if
\begin{equation}\label{eq:polygon-concavity-conditions}
 y_1=y_n=0,\qquad y_i\ge 0\quad (2\le i\le n-1),
\end{equation}
and
\begin{equation}\label{eq:slope-monotonicity}
\frac{y_{i+1}-y_i}{x_{i+1}-x_i}
\le
\frac{y_i-y_{i-1}}{x_i-x_{i-1}},
\qquad i=2,\dots,n-1.
\end{equation}
Equivalently, $\rho$ is a nonnegative concave piecewise-linear function on $[x_1,x_n]$, extended by zero outside this interval.
\end{theorem}

\begin{proof}
Let $P\subset\RR^2$ be a compact convex polygon and let
\[
\rho(x)=\operatorname{length}\{y:(x,y)\in P\}
\]
be its vertical slice length.  By the one-dimensional Brunn--Minkowski inequality applied to the vertical fibres, $\rho$ is concave on the projection interval of $P$.  Since $P$ is a polygon, $\rho$ is piecewise linear, and it vanishes at the two endpoints of the projection interval.  Thus the values $y_i=\rho(x_i)$ satisfy \eqref{eq:polygon-concavity-conditions} and the monotonicity of slopes \eqref{eq:slope-monotonicity}.

Conversely, suppose that $\rho$ is nonnegative, concave and piecewise linear.  Define
\[
P_\rho:=\{(x,y)\in\RR^2: x\in[x_1,x_n],\ -\rho(x)/2\le y\le \rho(x)/2\}.
\]
This set is convex: if $(x,y),(x',y')\in P_\rho$ and $0\le\lambda\le1$, then
\[
|\lambda y+(1-\lambda)y'|
\le \lambda\rho(x)/2+(1-\lambda)\rho(x')/2
\le \rho(\lambda x+(1-\lambda)x')/2
\]
by concavity of $\rho$.  Since $\rho$ is piecewise linear, $P_\rho$ is a compact convex polygon.  Its vertical slice at $x$ has length exactly $\rho(x)$, so its pushforward density is $\rho$.
\end{proof}

\begin{corollary}[The literal cone for polygons]\label{cor:literal-cone-polygons}
For $d=2$ the literal single-polytope cone inside $L_{2,\cX}$ is the polyhedral cone of nonnegative concave piecewise-linear densities.  In the nodal-value coordinates $(y_1,\dots,y_n)$ it is cut out by the linear equalities and inequalities \eqref{eq:polygon-concavity-conditions}--\eqref{eq:slope-monotonicity}.
\end{corollary}

\begin{proof}
This is just Theorem~\ref{th:single-polygon-realization} written in the coordinates given by the values of the density at the knots.
\end{proof}

For higher-dimensional polytopes the same argument gives a useful necessary condition, and also a large explicit sufficient family.

\begin{proposition}[Brunn--Minkowski obstruction and a realizable subfamily]\label{prop:BM-obstruction}
Let $d\ge2$ and let $P\subset\RR^d$ be a compact convex $d$-polytope.  If
\[
\rho(x)=\operatorname{vol}_{d-1}\{z\in P:z_1=x\}
\]
is its slice-volume density, then $\rho^{1/(d-1)}$ is concave on the projection interval of $P$.
Conversely, if $g$ is a nonnegative concave piecewise-linear function on an interval and $K\subset\RR^{d-1}$ is a compact convex $(d-1)$-polytope containing the origin, then
\[
P_{g,K}:=\{(x,u): x\in I,\ u\in g(x)K\}
\]
is a compact convex $d$-polytope and its projection density is
\[
\rho(x)=\operatorname{vol}_{d-1}(K)g(x)^{d-1}.
\]
\end{proposition}

\begin{proof}
The concavity of $\rho^{1/(d-1)}$ is the Brunn--Minkowski theorem applied to the fibres of $P$ over two points of the projection interval. For the converse, let $u\in g(x)K$ and $u'\in g(x')K$. Write $u=g(x)a$ and $u'=g(x')a'$ with $a,a'\in K$, with the evident interpretation when one of the values of $g$ is zero. Since $0\in K$ and $K$ is convex, the vector
\[
\frac{\lambda g(x)}{\lambda g(x)+(1-\lambda)g(x')}a+
\frac{(1-\lambda)g(x')}{\lambda g(x)+(1-\lambda)g(x')}a'
\]
belongs to $K$ whenever the denominator is nonzero. Concavity gives
\[
\lambda g(x)+(1-\lambda)g(x')\leq g(\lambda x+(1-\lambda)x'),
\]
and the inclusion $0\in K$ implies $sK\subset tK$ for $0\le s\le t$. Hence the convex combination of $(x,u)$ and $(x',u')$ belongs to $P_{g,K}$. Since $g$ is piecewise linear and $K$ is a polytope, $P_{g,K}$ is a polytope. The fibre over $x$ is the homothetic copy $g(x)K$, whose volume is $\operatorname{vol}_{d-1}(K)g(x)^{d-1}$.
\end{proof}

\section{Compatibility of different projections}\label{sec:projection-compatibility}

The second issue mentioned in the concluding remarks is the interaction among one-dimensional projections in different directions.  The following elementary observations separate what is already forced by the one-dimensional data from the genuinely multidirectional part of the inverse problem.

\begin{proposition}[Basic compatibility conditions]\label{prop:projection-compatibility}
Let $\mu$ be a compactly supported finite positive measure on $\RR^d$, and for every unit vector $u$ let $\mu_u$ be the pushforward of $\mu$ under $z\mapsto u\cdot z$.  Then the following hold.
\begin{enumerate}
\item The zeroth moment of $\mu_u$ is independent of $u$ and equals the total mass of $\mu$.
\item The first directional moment is linear in $u$:
\[
\int t\,d\mu_u(t)=u\cdot b,
\qquad b:=\int z\,d\mu(z).
\]
\item The second directional moment is a quadratic form in $u$:
\[
\int t^2\,d\mu_u(t)=u^T S u,
\qquad S:=\int zz^T\,d\mu(z),
\]
where $S$ is positive semidefinite.
\item More generally, the $r$-th directional moment is the evaluation on the diagonal $u=\cdots=u$ of the symmetric moment tensor
\[
T_r:=\int z^{\otimes r}\,d\mu(z).
\]
\end{enumerate}
Consequently, arbitrary independently chosen univariate higher Prony data in several directions need not be compatible with a common measure, and hence need not come from a common polytope.
\end{proposition}

\begin{proof}
Each assertion follows by expanding
\[
\int t^r\,d\mu_u(t)=\int (u\cdot z)^r\,d\mu(z).
\]
The cases $r=0,1,2$ give the displayed mass, barycentre and second-moment matrix conditions.  The general case is the corresponding symmetric tensor identity.
\end{proof}

\begin{proposition}[All directional projections determine the measure]\label{prop:cw}
The family of all one-dimensional pushforwards $\{\mu_u:u\in S^{d-1}\}$ determines the compactly supported measure $\mu$ uniquely.  In particular, it determines a convex polytope $P$ uniquely up to sets of Lebesgue measure zero when $\mu$ is Lebesgue measure restricted to $P$.
\end{proposition}

\begin{proof}
This is the standard Cramer--Wold/Fourier argument; see also the tomography viewpoint in \cite{Gardner}. The Fourier transform of $\mu_u$ is
\[
\widehat{\mu_u}(s)=\int e^{ist}\,d\mu_u(t)=\int e^{is u\cdot z}\,d\mu(z)=\widehat\mu(su).
\]
As $s\in\RR$ and $u\in S^{d-1}$ vary, the vectors $su$ fill all of $\RR^d$.  Hence the family of one-dimensional pushforwards determines the full Fourier transform $\widehat\mu$, and therefore determines $\mu$.  If $\mu$ is the restriction of Lebesgue measure to a convex body, then equality of measures implies equality of the bodies up to a null set; for compact convex polytopes this is the natural uniqueness statement.
\end{proof}

\subsection{Finite ambiguity from two generic projections}

We now formulate a simple rigidity statement illustrating the additional
constraints imposed by several projection directions.

Assume that $\mu$ is an atomic probability measure in $\RR^2$ supported at
$n$ distinct points
\[
p_i=(x_i,y_i),
\qquad i=1,\dots,n,
\]
with nonzero weights $w_i$.

For a direction $u=(u_1,u_2)$, the projected atoms are located at
\[
\lambda_i(u)=u_1x_i+u_2y_i.
\]
For generic $u$, the projected values are pairwise distinct, and the
classical Prony reconstruction recovers the unordered set
\[
\{\lambda_1(u),\dots,\lambda_n(u)\}
\]
together with the corresponding weights.

The main difficulty is therefore not reconstruction in a single direction
(which is governed by the KSS-Hankel conditions), but rather the matching
problem between several directions.

\begin{proposition}\label{pr:finiteambiguity}
Let $u,v\in\RR^2$ be linearly independent generic directions.  Then the
pair of projected datasets
\[
\{\lambda_i(u)\}_{i=1}^n,
\qquad
\{\lambda_i(v)\}_{i=1}^n
\]
determines the support configuration
\[
\{(x_i,y_i)\}_{i=1}^n
\]
up to at most finitely many possibilities.
\end{proposition}

\begin{proof}
Since $u$ and $v$ are linearly independent, the map
\[
T:\RR^2\to\RR^2,
\qquad
p\mapsto (\langle p,u\rangle,\langle p,v\rangle)
\]
is invertible.

If one knew the correct matching between the projected coordinates in the
two directions, then each point $p_i$ would be uniquely reconstructed by
solving a nondegenerate $2\times2$ linear system.

Thus all ambiguities arise solely from permutations matching the
$u$-projections with the $v$-projections.  Since there are only finitely
many such permutations, one obtains only finitely many candidate
configurations.
\end{proof}

\begin{remark}
Although elementary, the proposition clarifies the geometric nature of the multidirectional inverse problem.  It shows that
the multidirectional inverse problem has a fundamentally different nature
from the one-directional KSS theory: the essential issue becomes the
compatibility of matchings between several projected Prony systems.
\end{remark}

\subsection{A codimension estimate for compatible directional data}

We next give a quantitative form of the compatibility phenomenon.

Fix directions
\[
u_1,\dots,u_N\in\RR^d
\]
and consider directional moments up to order $R$:
\[
M_r(u_j),
\qquad
0\le r\le R,\quad 1\le j\le N.
\]

For fixed $r$, Proposition~\ref{pr:rank} shows that the possible vectors
\[
(M_r(u_1),\dots,M_r(u_N))
\]
belong to a vector space of dimension at most
\[
\binom{d+r-1}{r}.
\]

Hence one obtains the following.

\begin{theorem}\label{th:codim}
Assume that the directions $u_1,\dots,u_N$ are generic in $\RR^d$.
Then the space of compatible directional moment data up to order $R$ has
codimension
\[
\sum_{r=0}^R
\max\left\{
0,\,
N-\binom{d+r-1}{r}
\right\}
\]
inside the ambient space of all collections
\[
\{M_r(u_j)\}_{0\le r\le R,\ 1\le j\le N}.
\]
\end{theorem}

\begin{proof}
For each fixed order $r$, the directional moments arise from evaluations of
a homogeneous polynomial of degree $r$ in $d$ variables.  The dimension of
this space equals
\[
\binom{d+r-1}{r}.
\]
For generic directions, the evaluation map has maximal rank, hence the
space of admissible vectors has codimension
\[
\max\left\{
0,\,
N-\binom{d+r-1}{r}
\right\}.
\]
Summing over $r=0,\dots,R$ gives the result.
\end{proof}

\begin{corollary}
For fixed dimension $d$ and fixed truncation order $R$, sufficiently many
generic projection directions necessarily satisfy nontrivial compatibility
relations beyond the KSS Hankel conditions in each direction separately.
\end{corollary}

\begin{remark}
The codimension estimate above is purely linear and therefore only captures
the first layer of multidirectional compatibility.  One expects additional
nonlinear constraints coming from positivity, convexity, and common support
conditions.
\end{remark}

\section{Beyond KSS: compatibility of several axial moment sequences}

\noindent
Kohn--Sturmfels--Shapiro \cite{KSS} completely describe the algebraic relations associated with a single projection direction. In contrast, the present section studies additional compatibility phenomena arising from several directions simultaneously.

The results of Kohn--Sturmfels--Shapiro describe the algebraic relations
among moments obtained from a \emph{single} projection direction.  In
particular, for a fixed direction $u\in \RR^d$, the sequence
\[
m_r(u)=\int_{\RR^d} \langle x,u\rangle^r\, d\mu(x)
\]
satisfies the classical Hankel determinantal relations associated with a
Prony system.  One may ask whether independent KSS conditions in several
directions are sufficient for the existence of a common underlying measure.
The answer is negative.

The main new feature in the multidirectional setting is the appearance of
\emph{compatibility constraints} between different directional moment
sequences.  These constraints are invisible from the viewpoint of any single
Hankel matrix.

Let
\[
M_r(u)=\int_{\RR^d} \langle x,u\rangle^r\, d\mu(x).
\]
Expanding the power gives
\[
M_r(u)=
\sum_{|\alpha|=r}
\binom{r}{\alpha}
u^\alpha m_\alpha,
\]
where
\[
m_\alpha=\int_{\RR^d} x^\alpha\, d\mu(x)
\]
are the ordinary multivariate moments.

Thus for each fixed $r$, the directional moments form a homogeneous
polynomial of degree $r$ in the direction variables.

\begin{proposition}\label{pr:compatibility}
Let $\{s_r^{(j)}\}_{r\ge0}$ be moment sequences assigned to directions
$u_1,\dots,u_N\in\RR^d$.  A necessary condition for the existence of a
compactly supported measure $\mu$ satisfying
\[
s_r^{(j)}=M_r(u_j)
\]
for all $r,j$ is the existence, for every $r\ge0$, of a homogeneous
polynomial $P_r(u)$ of degree $r$ such that
\[
P_r(u_j)=s_r^{(j)},\qquad j=1,\dots,N.
\]
Equivalently, the collections $\{s_r^{(j)}\}_j$ must lie in the image of
the Veronese evaluation map
\[
\mathrm{Sym}^r(\RR^d)^*
\longrightarrow
\RR^N,
\qquad
P\mapsto (P(u_1),\dots,P(u_N)).
\]
\end{proposition}

\begin{proof}
The formula above shows that
\[
P_r(u):=M_r(u)
\]
is a homogeneous polynomial of degree $r$.  Hence the values
$s_r^{(j)}$ must arise as evaluations of one and the same polynomial
at the prescribed directions.
\end{proof}

\begin{remark}
For one direction this condition is vacuous and one recovers precisely the
KSS picture.  Starting already from two or three directions, nontrivial
cross-relations appear.
\end{remark}

We now illustrate this phenomenon in the planar case.

\begin{example}
Consider a compactly supported planar probability measure $\mu$ and let
\[
a_r=M_r(e_x),\qquad
b_r=M_r(e_y),
\]
be the axial moments in the coordinate directions.
Then
\[
M_2(u,v)=u^2 a_2+2uv\,m_{11}+v^2 b_2.
\]
Hence the second directional moments along \emph{all} directions are
determined not only by the two axial sequences but also by the mixed moment
$m_{11}$.

In particular, arbitrary independent choices of two admissible Hankel
moment sequences $\{a_r\}$ and $\{b_r\}$ do not determine a planar measure:
one still has to solve the compatibility problem for the mixed moments.
\end{example}

The same phenomenon persists in higher order.  For example,
\[
M_3(u,v)
=
u^3 m_{30}
+
3u^2v\,m_{21}
+
3uv^2\,m_{12}
+
v^3 m_{03},
\]
so knowledge of the third moments in finitely many directions imposes
linear constraints on the mixed moments.

\begin{proposition}\label{pr:rank}
Fix $r\ge0$.  The space of degree-$r$ directional moment functions
\[
u\mapsto M_r(u)
\]
has dimension
\[
\binom{d+r-1}{r}.
\]
Consequently, any collection of directional moment data in more than
$\binom{d+r-1}{r}$ generic directions satisfies nontrivial linear
relations.
\end{proposition}

\begin{proof}
The map
\[
u\mapsto M_r(u)
\]
is a homogeneous polynomial of degree $r$ in $d$ variables.  The vector
space of such polynomials has the stated dimension.
\end{proof}

Proposition~\ref{pr:rank} gives a family of universal linear relations
between directional moment measurements.  These relations are independent
from the Hankel determinantal relations of Kohn--Sturmfels--Shapiro and
therefore constitute genuinely new constraints in the multidirectional
inverse problem.

This suggests the following natural problem.

\begin{problem}
Describe the algebraic variety of collections of directional moment
sequences
\[
\{M_r(u_j)\}_{r,j}
\]
arising from compactly supported measures in $\RR^d$.
In particular, determine the compatibility relations between the KSS
determinantal conditions associated with different directions.
\end{problem}

The multidirectional compatibility problem appears to be substantially more
rigid than the one-directional situation studied in \cite{KSS}.  Already
for polygons in the plane one obtains nontrivial matching problems between
the projected vertices in different directions, closely related to inverse
tomography and phase retrieval.

At present, the paper should probably be viewed as a hybrid between a
research note and a conceptual reorganization of the Kohn--Sturmfels--
Shapiro theory from the viewpoint of inverse problems and Prony systems.
The genuinely new part is not the one-directional Hankel geometry itself,
which is already contained in \cite{KSS}, but rather:
\begin{itemize}
\item the explicit Prony reconstruction interpretation;
\item positivity and spline-cone considerations;
\item the distinction between spline realizability and realizability by a
single convex polytope;
\item the multidirectional compatibility constraints introduced in the last
section.
\end{itemize}

The latter direction appears to be the most promising for further
development.  In particular, it would be interesting to understand the full
algebraic structure of collections of directional moment sequences arising
from one common measure and to compare this structure with classical
tomographic reconstruction problems.

\end{document}